\documentclass[10pt]{amsart}

\usepackage{amsmath,amssymb,amsthm,amsrefs,a4wide}
\usepackage{latexsym}
\usepackage{indentfirst}
\usepackage{graphicx}
\usepackage{placeins}
\usepackage{booktabs}
\usepackage{algorithm}
\usepackage{algorithmic}
\usepackage{multirow}
\usepackage{subcaption}
\usepackage{authblk}
\usepackage[foot]{amsaddr}
\usepackage{graphicx,color}
\usepackage{diagbox}

\theoremstyle{plain}
\newtheorem{theorem}{Theorem}

\theoremstyle{definition}

\theoremstyle{remark}

\newtheorem{assumption}{Assumption}

\newtheoremstyle{cited}%
  {3pt}
  {3pt}
  {\itshape}
  {}
  {\bfseries}
  {.}
  {.5em}
  {\thmname{#1} \thmnumber{#2} \thmnote{\normalfont#3}}

\theoremstyle{cited}

\DeclareMathOperator{\tr}{tr}

\newcommand{\iu}{{i\mkern1mu}}

\newcommand{\bbm}{\begin{bmatrix}}
\newcommand{\ebm}{\end{bmatrix}}

\newcommand{\bbR}{\mathbb{R}}

\newcommand{\grad}{\nabla}
\newcommand{\I}{\mathcal{I}}
\newcommand{\J}{\mathcal{J}}

\newcommand\figref[1]{Figure\ \ref{#1}}
\newcommand\secref[1]{Section\ \ref{#1}}

\title{Convex Relaxation for Fokker-Planck}

\author{Yian Chen}
\address{Department of Statistics, University of Chicago}

\author{Yuehaw Khoo}
\address{Department of Statistics, University of Chicago}

\author{Lek-Heng Lim}
\address{Department of Statistics, University of Chicago}

\email{yianc@uchicago.edu, ykhoo@uchicago.edu,lekheng@uchicago.edu}

\begin{document}

\begin{abstract}
    We propose an approach to directly estimate the moments or marginals for a high-dimensional equilibrium distribution in statistical mechanics, via solving the high-dimensional Fokker-Planck equation in terms of low-order cluster moments or marginals. With this approach, we bypass the exponential complexity of estimating the full high-dimensional distribution and directly solve the simplified partial differential equations for low-order moments/marginals. Moreover, the proposed moment/marginal relaxation is fully convex and can be solved via off-the-shelf  solvers. We further propose a time-dependent version of the convex programs to study non-equilibrium dynamics. We show the proposed method can recover the meanfield approximation of an equilibrium density.  Numerical results are provided to demonstrate the performance of the proposed algorithm for high-dimensional systems.
\end{abstract}

\maketitle
\section{Introduction}

A central question in statistical mechanics is the determination of the equilibrium distribution \cite{pathria2016statistical} from a stochastic differential equation (SDE). 
Standard approaches include applying Monte-Carlo method to simulate an SDE or variational inference method that minimizes an entropic regularized energy functional \cites{blei2017variational,wu2019solving}. While these approaches have various successes, a less explored route to study the equilibrium is via solving the partial differential equation (PDE) associated with the SDE, which characterizes the evolution of the distribution. Obtaining the equilibrium distribution can be done by looking at the time-independent version of such a PDE.  However, for a $d$-dimensional SDE, one has to solve a $d$-dimensional PDE which results in the curse-of-dimensionality. While one can solve the resulting PDE via assuming various low-complexity ansatz such as a neural-network or tensor-network for the solution, obtaining the solution requires optimizing a non-convex objective function over high-dimensional spaces \cites{han2018solving,yu2018deep,chen2023committor,chertkov2021solution}. In this note we take a different approach to it via using a convex-relaxation strategy to obtain the moments of the distribution. While our method could potentially be used to study the transient behaviour of general SDEs, we demonstrate the success of our method for the case of overdamped Langevin dynamics \cite{schlick2010molecular}, where we solve for the equilibrium distribution via solving the Fokker-Planck equation (FPE).

\subsection{Fokker-Planck equation}
In the rest of the paper, we assume we are given a confining potential function (see, e.g., \cite[Definition 4.2]{bhattacharya2009stochastic}) $V:\mathbb{R}^d\rightarrow \mathbb{R}$ that has a pairwise structure, i.e.
\begin{equation}
    V(x) = \sum_{i=1}^d V_i(x_i) + \sum^d_{i\neq j=1} V_{ij}(x_i,x_j).
    \label{eq:pairwise_potential}
\end{equation}
The pairwise potential function is ubiquitous in a wide range of physical, chemical systems and molecular dynamics simulations \cites{halliday2013fundamentals,rohrlich1989paradox,morse1929diatomic,lennard1924determination}.

The goal is to study the equilibrium distribution $\rho^\star(x) \propto \exp(-V(x)/T)$ of the following overdamped Langevin processes,
\begin{align}
    d x_t = -\nabla V(x_t) \, dt + \sqrt{2\beta^{-1}} \, dW_t,
    \label{eq:langevin}
\end{align}
where $x_t \in \Omega\subset\bbR^d$ is the state of stochastic system, $\beta = 1/T$ where $T$ is the temperature, and $W_t$ is a $d$-dimensional Wiener process. 

Standard approaches (for example in molecular dynamics) use a long time Markov-chain Monte-Carlo (MCMC) simulation to study the equilibrium distribution \cite{kikuchi1991metropolis}. For $V$ with a complicated landscape and high-dimensionality, the mixing-time of an MCMC can be long.

Deterministic approaches include the solution to the following variational problem
\begin{equation}\label{eq:VI}
    \min_{\rho\geq 0, \int \rho(x) dx = 1} \langle \rho, V\rangle + \frac{1}{\beta} \langle \rho, \log \rho \rangle
\end{equation}
where the minimizer is in fact the equilibrium distribution $\rho^\star$. Since $\rho:\mathbb{R}^d\rightarrow \mathbb{R}$ is a high-dimensional function, various low-complexity ansatz such as meanfield \cite{blei2017variational}, bethe approximation \cite{katsura1974bethe}, generative models \cite{kobyzev2020normalizing}, and tensor-networks \cite{dolgov2020approximation} has been used to represent $\rho$. These approaches either leads to a non-convex optimization domain, or approximations to the entropy term \cite{risteski2016calculate}. 

On the other hand, the equilibrium distribution satisfies the $d$-dimensional stationary FPE:
\begin{equation}\label{eq:FPE}
    L\rho(x) = -\frac{1}{\beta} \Delta \rho(x)- \grad \cdot (\rho(x)  \grad V(x)) = 0.
\end{equation}
There are many approaches available to solve the FPE. The traditional finite difference and finite element methods work well for low-dimensional problems but they both scale exponentially with the number of dimensions. To circumvent the curse of dimensionality, researchers propose to pose various low-complexity ansatz on the solution of FPE to control the growth of parameters. For example, \cites{han2018solving,yu2018deep,zhai2022deep} propose to parametrize the unknown PDE solution with deep neural networks and optimize their underlying stochastic differential equations or variational problems instead. \cites{ambartsumyan2020hierarchical,chen2021scalable,chen2023scalable,ho2016hierarchical} approximates the differential operators with data-sparse hierarchical matrices. \cites{kazeev2014direct} parametrizes the PDE solution using tensor networks \cites{orus2014practical,oseledets2011tensor, peng2023generative} where a similar approach \cite{chen2023committor} is proposed for the adjoint equation. These parametric models effectively controls the complexity of the problem but again, obtaining them requires the use of non-convex optimization.

\subsection{Our contributions}

In this paper, we propose a novel approach to obtain the moments of the equilibrium distribution, via solving for the FPE in terms of the low order marginals or cluster moments, which circumvent the curse-of-dimensionality. The optimization problems we proposed are fully convex, hence avoiding the difficulty of optimizing a globally nonconvex objective in neural network and tensor network approaches. While this has similarities to the convex hierarchies proposed to solve \eqref{eq:VI}, due to the linearity of the FPE, we completely bypass the need to approximate the entropy term in terms of the moments, which often leads to a non-convex cost (as in the case of Bethe approximation). We show theoretically that a meanfield approximation can be recovered when being used to obtain the 1-marginals of the equilibrium density. Hence our approach can be regarded as a convex-relaxation approach for obtaining a meanfield solution to the meanfield FPE (or McKean-Vlasov type equations) which is nonlinear \cite{funaki1984certain}. 

We view this note as extending methods described in \cite{lasserre2009moments} for solving 1D PDE to high-dimensional cases. We show via analysis and numerical experiments that choices convex-hierarchies (cluster moment and marginal relaxations) in similar spirit with cluster expansion \cite{brydges1976cluster} or marginal relaxation \cites{sontag2007new,peng2012approximate,chen2020multiscale} can be used to handle high-dimensional systems that are weakly correlated. We note that this work is substantially different from \cite{raj2023efficient} where the non-negative density is parameterized as a sum-of-squares polynomial. For us, only the moments of the distribution is solved for, which allows us to efficiently characterize a degenerate distribution. 

\subsection{Organization}
The rest of the note is organized as follows.  In Section~\ref{sec:pde_hierarchy}, we propose various convex hierarchies to solve \eqref{eq:FPE}. 
In Section~\ref{sec:time_dependent}, we extend the proposed method to time-dependent problem. In Section~\ref{sec:analysis}, we analyze the recovery property of recovering the meanfield approximation to an equilibrium density from the proposed convex hierarchies. In Section~\ref{sec:numerical}, we study the proposed method in a few numerical experiments. We then conclude in Section~\ref{sec:conclusion}.


\section{PDE hierarchies for solving the FPE}
\label{sec:pde_hierarchy}

In this section, we present two formulations: cluster moments relaxation (\secref{sec:moment_relaxation}) and marginal relaxation (\secref{sec:marginal_relaxation}) for solving the FPE. 

\subsection{Cluster moment relaxation}
\label{sec:moment_relaxation}

Let $\Omega\in \mathbb{R}$ be the computational domain for a single variable. Let $\mathcal{T}:=\{t\ \vert \ t:\Omega^d\rightarrow \mathbb{C}\}$ be a finite dimensional function space. Furthermore, let $\mathcal{Q}:=\{q\ \vert \ q:\Omega^d\rightarrow \mathbb{C}, q\geq 0\}$ be a finite dimensional space of non-negative functions.

Suppose we want to solve 
\begin{equation} 
L\rho = 0,\quad \rho\geq 0,\quad \int \rho(x) dx = 1.
\end{equation}
While this is a linear PDE in $\rho$, it requires exponential number of basis in $d$ to discretize $\rho$. To overcome the dimensionaliy of the problem, we solve a relaxed version of the PDE problem via enforcing necessary conditions for $L\rho = 0$ and $\rho\geq 0$ via test functions $t\in \mathcal{T}$ and $q\in \mathcal{Q}$. More specifically, we let:
\begin{eqnarray}\label{eq:adjoint problem}
&\ &\langle L^*t, \rho\rangle = \langle t, L\rho\rangle = 0, \quad \forall t \in \mathcal{T}, \cr
&\ &\langle q, \rho \rangle  \geq 0,\quad \forall q \in \mathcal{Q}, \cr 
&\ &\int \rho(x) dx = 1.
\end{eqnarray}
where the adjoint operator $L^*$ apply to $t\in \mathcal{T}$ is
\begin{eqnarray}
L^* t(x) &=& -\frac{1}{\beta}\Delta t(x) + \nabla t(x) \cdot \nabla V(x) \cr 
\end{eqnarray}
The adjoint operator is derived using the fact that we impose the boundary condition
\begin{equation}
    \rho(x),\ \grad \rho(x) = 0,\quad x\in \partial \Omega.
    \label{eq:boundary}
\end{equation}
This is a reasonable assumption especially when $V$ is confining. We now discuss choices of space for $\mathcal{T}$ and  $\mathcal{Q}$ that lead to tractable convex program. 

\subsubsection{Cluster basis}

Let $\{\phi_j:\Omega\subset \mathbb{R}\rightarrow \mathbb{C} \}_{j=1}^n$ be a set of single variable basis. Let $\mathcal{F}_K$ be the space of one-body to $K$-body functions where  $\mathcal{F}_K = \text{span}\{\phi_{j_1}(x_{i_1})\cdots \phi_{j_K}(x_{i_K})\ \vert \ \J:=(j_1,\ldots j_K)\in [n]^K, \I:=(i_1\ldots,i_K)\in  {d \choose K}\}$. For our purpose, we choose $\{\phi_j\}_{j=1}^n$ to be the monomial or Fourier basis and $\phi_1 \propto 1$. In what follows, we often use the notations $x_{\I} := (x_{i_1},\cdots,x_{i_K})$ and  $\phi_{\J}(x_\I):=\phi_{j_1}(x_{i_1})\cdots \phi_{j_K}(x_{i_K})$. We also let the set of sum-of-squares functions be $\mathcal{S}_K = \{ [\phi_\J(x_\I)]_{\I\J} A  [\phi_\J(x_\I)]_{\I\J}^T \ \vert\ A\succeq 0, \phi_\J \in  \mathcal{F}_K\}$ where the positive-semidefinite matrix $A$ has size ${d \choose K}n^K\times {d \choose K}n^K$.

We choose $\mathcal{T} = \mathcal{F}_K$ and furthermore $\mathcal{Q} = \mathcal{S}_K$.
In this case, applying the adjoint operator $L^*$ to $t\in \mathcal{T}$ gives
\begin{eqnarray}
(L^* t)(x_{\I}) &=& -\frac{1}{\beta}\Delta t(x_{\I}) + \nabla t(x_{\I}) \cdot \nabla V(x) \cr 
&=& \sum^K_{k=1} -\frac{1}{\beta}\frac{\partial^2 t(x_{\I}) }{\partial x_{i_k}^2} + \sum^K_{k=1} \frac{\partial t(x_{\I}) }{\partial x_{i_k}} \left(\frac{\partial V_{i_k}(x_{i_k})}{\partial x_{i_k}} + \sum_{i_l:i_l\neq i_k} \frac{\partial V_{i_k,i_l}(x_{i_k},x_{i_l})}{\partial x_{i_k}}\right),\quad \I\in {d \choose K} .
\end{eqnarray}

Suppose $V_{k,l}(x_k,x_l)$ can be expanded by $\{\phi_{i_k}(x_k)\phi_{i_k}(x_l)\}^{n}_{i_k,i_l=1}$ up to sufficient numerical precision. Furthermore, we pick $t=\phi_{\J}(x_\I)$, $\J\in [n]^K,\ \I \in {d \choose K}$ where $K\geq 1$. In this way, \eqref{eq:adjoint problem} can be written in terms of the following matrix variable
\begin{equation}\label{eq:pair moment}
M_{\I\J,\I'\J'} = \int \phi_\J(x_\I)\phi_{\J'}(x_{\I'}) \rho(x) dx
\end{equation}
that scales as ${d \choose K}n^K\times {d \choose K}n^K$, since $\{\phi_\J(x_\I)\phi_{\J'}(x_{\I'})\}=\mathcal{F}_K\otimes \mathcal{F}_K$ and this includes
\begin{equation}
 -\frac{1}{\beta}\frac{\partial^2 \phi_\J(x_{\I}) }{\partial x_{i_k}^2}  , \quad  \frac{\partial \phi_\J(x_{\I}) }{\partial x_{i_k}} \frac{\partial V_{i_k}(x_{i_k})}{\partial x_{i_k}} ,  \quad \frac{\partial \phi_\J(x_{\I}) }{\partial x_{i_k}}  \sum_{i_l:i_l\neq i_k} \frac{\partial V_{i_k,i_l}(x_{i_k},x_{i_l})}{\partial x_{i_k}} ,
\end{equation}
 in $\mathcal{F}_K$, $\mathcal{F}_K\otimes \mathcal{F}_1$, $\mathcal{F}_K\otimes \mathcal{F}_2$ respectively. More precisely, the main optimization problem derived from \eqref{eq:adjoint problem} becomes
\begin{eqnarray}
&\ &l_{\I \J}(M) = 0,\quad \J\in [n]^K,\ \I \in {d \choose K}\cr 
&\ &M\succeq 0\cr
&\ &M_{\I=\emptyset \J=0,\I'=\emptyset \J'=0} = 1
\label{eq:moment_constraint}
\end{eqnarray}
where $l_{\I \J}(M)=0$ is the reformlutation of 
$$\int (L^* \phi_\J)(x_\I) \rho(x)dx = 0$$ in terms of the moment matrix $M$, and the positive semidefinite constraints $M \succeq 0$ comes from the fact that   $$\int [\phi_\J(x_\I)]_{\I\J} A  [\phi_\J(x_\I)]_{\I\J}^T\rho(x)dx \geq 0$$ for any $A\succeq 0$.

\subsubsection{Symmetric cluster basis}

For interacting particle systems, often one needs to deal with a potential $V$ that is symmetric, i.e. $V(x_1,\ldots,x_d)= V(x_{\sigma(1)},\ldots,x_{\sigma(d)})$ for any $\sigma$ in the symmetric group $\text{Sym}(d)$. In this case one can let $\mathcal{T}=\mathcal{F}_{S,K}$ where $\mathcal{F}_{S,K} = \{\frac{1}{d!} \sum_{\sigma\in \text{Sym}(d)} f(x_{\sigma(1)},\ldots,f_{\sigma(d)})\ \vert \ f\in \mathcal{F}_K\}$, which consists of symmetric functions in  $\mathcal{F}_{K}$. This allows us to span $\mathcal{F}_{S,K}$ with $n^K$ basis instead of ${n \choose K}n^K$ basis, hence achieving a significant complexity gain. Similarly, one can let the set of nonnegative functions $\mathcal{Q} = \mathcal{S}_{S,K}$ where $\mathcal{S}_{S,K} = \{ [\phi_\J(x)]_{\J} A  [\phi_\J(x)]_{\J}^T \ \vert\ A\succeq 0, \phi_\J \in  \mathcal{F}_{S,K}\}$, i.e. the symmetric sum-of-squares functions.

\subsection{Marginal relaxation}
\label{sec:marginal_relaxation}

Another possible convex relaxation to solve \eqref{eq:FPE} is via discretized low-order marginals. This can be useful, when $V$ exhibits singularities (for example when pairiwse interactions $V_{ij}$ is coulombic). Let $\rho_\I$ be the $\I$-th marginal of $\rho$:
\begin{align}
    & \rho_\I(x_\I) := \int \rho(x) dx_{[d]\setminus \I},
\end{align}
where $\I\in {d \choose K}$. A reduced order PDE in terms of the marginal can be derived by partial integrations of the FPE:
\begin{equation}
    \int \left[-\frac{1}{\beta} \Delta \rho(x)- \grad \cdot (\rho(x)  \grad V(x))\right] dx_{[d]\setminus \I} = 0
\end{equation}
which gives
\begin{equation}
    -\frac{1}{\beta} \sum_{k=1}^K \frac{\partial^2\rho_\I(x_\I)}{\partial x_{i_k}^2}-\sum_{k=1}^K \frac{\partial}{\partial x_{i_k}} \left(\sum_{j:j\in \I^c}\int  \rho_{\I\cup j}(x_\I,x_j)  \frac{\partial V_{i_kj}}{\partial x_{i_k}}(x_{i_k},x_j)dx_j\right) = 0.
    \label{eq:key}
\end{equation}
This gives equations for all $\rho_{\I}$ and $\{\rho_{\I\cup j}\}^d_{j=1}$. We further impose neccessary conditions that these marginals are derived from the same density $\rho$, i.e.
\begin{equation}\label{eq:local consistency}
\int \rho_{\I \cup j}(x_{\I \cup j}) dx_{j} = \rho_{\I}(x_{\I}),\  \rho_{\I\cup j}\geq 0,\ \forall  j\in [d], \ \I \cap j = \emptyset, \quad \int \rho_\I(x_\I)dx_\I = 1,\quad \forall \I \in {d \choose K}.
\end{equation}
Now to solve  \eqref{eq:key} along with \eqref{eq:local consistency}, we let $\tilde \rho_\I$ be a $n^K$ points discretization of $\rho_\I$ for each $\I$. Then $\frac{\partial^2\rho_\I(x_\I)}{\partial x_{i_k}^2}$ and $\frac{\partial \rho_{\I\cup j}}{\partial x_{i_k}}$ are discretized by a central and forward difference respectively in terms of $\tilde \rho_\I$ and $\tilde \rho_{ \I\cup j}$, and we get a discretized version of \eqref{eq:key}:
\begin{equation}\label{eq:discrete FPE marginal}
\tilde L_\I(\tilde \rho_\I, \{ \tilde \rho_{\I\cup j} \}_j) = 0.
\end{equation}
Along with a quadrature scheme for enforcing \eqref{eq:local consistency} one has a Sherali-Adams type linear programming hierarchy \cite{laurent2003comparison} with ${d \choose K+1}n^{K+1}$ number of variables ($K\geq 1$). For example, when discretizing the domain $\Omega^d$ as a uniform grid $X\subset \mathbb{R}^d$, one can use the quadrature scheme
\begin{equation}\label{eq:discretized marginal constraints}
\sum_{x_{j}\in X_{j}} \tilde \rho_{\I \cup j}(X_{\I},x_{j}) w_{x_{j}} = \tilde \rho_{\I}(X_{\I}),\  \tilde \rho_{\I\cup j}\geq 0,\ \forall  j\in [d],\ \I \cap j = \emptyset, \  \sum_{x_{\I}\in X_\I} \tilde \rho_\I(x_{\I})w_{x_{\I}}= 1,\  \forall \I \in {d \choose K}
\end{equation}
where $X_{\I}$ denotes the slice of grid points $X$ in $\I$-th dimension, and $w_{x_\I}$'s are the quadrature weights.  As in \cites{peng2012approximate,khoo2019convex,chen2020multiscale}, one can further  add semidefinite constraints to strengthen the convex-relaxation:
\begin{equation}
[G_{\I\I'}]_{\I\I'}\succeq 0\quad  \text{where}\ G_{\I\I'}=\tilde \rho_{\I\cup \I 
'}, \ G_{\I\I}-\text{diag}(G_{\I\I}) = 0,\ \text{diag}(G_{\I\I}) = \tilde \rho_\I
\label{eq:PSD_constraint}
\end{equation}
where $G$ has size ${d \choose K}n^K\times {d \choose K}n^K$. We remark that as in the case of cluster moments relaxation, symmetrization can also be done with marginals, leading to variable with size $n^K$, as in \cite{khoo2019convex}.

\subsection{Additional regularizations}
As we show later in Section~\ref{sec:analysis}, suppose the underlying equilibrium density $\rho^\star$ admits a meanfield or block meanfield type solution, i.e. $\rho(x) \approx \prod_{\mathcal{I}} \rho_\mathcal{I}(x_\mathcal{I})$ where $\{\I\}$ partitions $[d]$, the proposed convex programs can stably determine each $\rho_\mathcal{I}$. However, in generally one cannot hope to recover higher order marginals $\rho_{\I\cup j}$ or moments $M_{\I\J,\I'\J'}$ in \eqref{eq:pair moment} (moments of $\rho_{\I\cup \I'}$), since the total number of equations scales as ${d\choose K} = O(d^K)$, and the number of these higher order marginals or moments  scales as $O(d^{2K})$ or $O(d^{K+1})$ respectively. 

In order to reduce the number of variables, one can introduce penalties to regularize the solution such that it admits a meanfield type behavior. One can add the following cost to the convex programs in Section~\ref{sec:moment_relaxation} and \ref{sec:marginal_relaxation}: (1) Minimize the nuclear norm $\|M_{\I\J,\I'\J'}\|_*$ for each pair $(\I,\I')\in {d\choose K}\times {d\choose K}$ for moment relaxation. Minimize the nuclear norm $\|\tilde \rho_{\I\cap j}\|_*$ for each $(\I,j)\in {d\choose K}\times [d]$ for marginal relaxation, treating $\tilde \rho_{\I\cap j}$ as a $n^K\times n$ matrix. (2) Maximize the entropy of each $\tilde \rho_{\I\cap j}$ for each  $(\I,j)\in {d\choose K}\times [d]$ for marginal relaxation. For $\rho_{\I\cap j}$ with marginals being $\rho_{\I}$ and $\rho_{j}$, the maximum entropy solution is $\rho_{\I\cap j} = \rho_{\I}\rho_{j}$. In a similar vein, interior point methods \cite{potra2000interior} for convex programming always have a log barrier function, which serves as an entropic-like regularization for  $\tilde \rho_{\I\cap j}$.




\section{Time-dependent problems}
\label{sec:time_dependent}

In this section, we discuss possible extensions of our framework to study non-equlibrium dynamics with time-dependent Fokker-Planck equation
\begin{equation}\label{eq:time dep Fokker}
\frac{\partial \rho(x,\tau)}{\partial \tau} = -L\rho(x,\tau), \quad (x,\tau) \in \Omega^d \times [0,T]
\end{equation}
subject to initial on $\rho(x,0)$ and extra boundary conditions on $\rho(x,\tau)$. For example when studying first passage time \cites{wille2004new, artime2018first}, one would have the boundary conditions
\begin{align}\label{eq:time dependent bc}
    & \rho(x,0) = \delta(x-a),\cr
    & \rho(x,\tau) = 0,\quad x\in A\subset \Omega^d, \tau \in (0,T].
\end{align}
where particle starts at $\tau=0$ at $a\in \mathbb{R}^d$ and gets absorbed when hitting region $A$. We let $A$ be a semialgebraic set defined by
\begin{equation}
h(x)\leq 0,
\end{equation}
for example a ball or a half-space
\begin{equation}
h(x)=\vert x-b\vert^2-r^2\leq 0,\ b\in \mathbb{R}^d,\quad h(x) = c^T x\geq 0,\ c\in \mathbb{R}^d.
\end{equation}
Then the support of $\rho$ can be characterized by $h(x) \geq 0$.

\subsection{Time dependent cluster moment relaxation}
\label{sec:time_dependent_moment}

To apply convex relaxation to such a problem, we perform a direct discretization in time by approximating $\rho(x,t)$ as $\rho^1(x):=\rho(x,t_1),\ldots,\rho^m(x):=\rho(x,t_m)$, where $t_l = \frac{(l-1)T}{m-1}$. We again have an equation 
\begin{equation}
    \frac{\rho^{l+1}(x) - \rho^{l}(x)}{\delta \tau} + L\rho^l (x) = 0, \quad l=1,\ldots,m-1,
\label{eq:evolution_equation}
\end{equation}
where its convex relaxation can be obtained via testing against $t\in \mathcal{T}$, which gives 
\begin{eqnarray}\label{eq:adjoint problem time}
&\ & \left \langle t, \frac{\rho^{l+1} - \rho^{l}}{\delta \tau} + L\rho^l \right\rangle=0, \quad \forall t \in \mathcal{T},\ \forall l=1,\ldots, m-1\cr
&\ &\langle q, \rho^l\rangle  \geq 0,\quad \forall q \in \mathcal{Q}, \ \forall l=1,\ldots, m\cr 
&\ &\langle s h,\rho^l \rangle\geq 0, \quad \forall s\in \mathcal{Q},\ \forall l=1,\ldots, m\cr 
&\ & \langle w,\rho^1 \rangle =  \langle w,\delta(\cdot-a)\rangle,\quad \forall w\in \mathcal{T}, \ \forall l=1,\ldots, m\cr
&\ &\int \rho^1(x) dx = 1.
\end{eqnarray}

\subsection{Time dependent marginal relaxation}
\label{sec:time_dependent_marginal}

In this section, we present a different way to solve the time-dependent problem via marginal relaxations. We denote the $\I$-th marginal at time $l$ by
\begin{align}
   & \rho^l_\I(x_\I) := \int \rho(x, t_l) dx_{[d]\setminus \I}
\end{align}
and its disretization on the $\I$-th slice $X_\I\subset \Omega^{\vert \I\vert}$ of $X\subset \Omega^d$ as $\tilde \rho^l_\I$. We again look at the time discretized equation \eqref{eq:evolution_equation}, however, we integrate the equation in \eqref{eq:evolution_equation} with $\int\ \cdot\ dx_{[d]\setminus \I}$ and solve for $\rho^l_\I$ in terms of the discretized $\tilde \rho^l_\I$, while imposing constraints such as \eqref{eq:discretized marginal constraints} on $\tilde \rho^l_\I$ for each time $l$. As opposed to the discretized stationary equation \eqref{eq:discrete FPE marginal}, we have
\begin{align}
   &\frac{\tilde \rho^{l+1}_{\I} - \tilde \rho^{l}_{\I}}{\delta \tau} = -\tilde L_\I(\tilde \rho_\I, \{ \tilde \rho_{\I\cup j} \}_j) ,\ \notag\\
   &\forall l=1,\ldots, m.\notag\\
   &\sum_{x_{j}\in X_{j}} \tilde \rho_{\I \cup j}(X_{\I},x_{j}) w_{x_{j}} = \tilde \rho_{\I}(X_{\I}),\  \forall  j\in [d],\ \I \cap j = \emptyset,\ \forall l=1,\ldots, m \notag\\
   &\tilde \rho^l_\I\geq 0,\ \sum_{x_{\I}\in X_\I} \tilde \rho_\I(x_{\I})w_{x_{\I}}= 1,\  \forall \I \in {d \choose K},\ \forall l=1,\ldots, m,    \cr
   &[G^l_{\I\I'}]_{\I\I'}\succeq 0\   \text{where}\ G^l_{\I\I'}=\tilde \rho_{\I\cup \I '}, \ G^l_{\I\I}-\text{diag}(G^l_{\I\I}) = 0,\ \text{diag}(G^l_{\I\I}) = \tilde \rho_\I, \forall l=1,\ldots, m, \notag\\ \label{eq:key_evolving}
\end{align}
where $\tilde L_\I$ is defined as in \eqref{eq:discrete FPE marginal}. As for the initial and boundary conditions in \eqref{eq:time dependent bc}, we approximate them as
\begin{equation}
\tilde \rho^1_{\I}(X_\I) = f_0(X_\I),\quad  \tilde \rho_\I^l(X_\I \cap A_\I) = 0, \quad \forall l=2,\dots,m, \ \forall \I \in {d \choose K},
\end{equation}
where $f_0:\mathbb{R}^{\vert \I \vert} \rightarrow \mathbb{R}$ is a locally supported function that approximates the $\vert\I\vert$ dimensional delta function $\delta(\cdot-a_\I)$, and $a_\I$, $A_\I$ are the $\I$-th slice of the coordinate $a\in \Omega^d$ and set $A\subset \Omega^d$ respectively.

\section{Analysis}\label{sec:analysis}

In this section, we analyze the recovery property of the continuous marginal relaxation (Section~\ref{sec:marginal_relaxation}). We show that the proposed convex program can determine a \emph{meanfield} approximation to the equilibrium density $\rho^\star$. This is in stark contrast with determining the meanfield approximation from a nonlinear Mckean-Vlasov type equations \cite{funaki1984certain,gartner1988mckean}, since the equation we are solving is completely linear in terms of  variables on a convex domain. We believe the proof technique is generalizable to moment based relaxation (Section~\ref{sec:moment_relaxation}) and discretized marginal relaxation \eqref{eq:discretized marginal constraints} by suitably incorporating the error of approximating low-order marginals by moments or interpolations.

\subsection{Preliminaries and assumptions}
For simplicity, we assume that each variable is on a periodic domain $\Omega = [0,1]$. The potential $V\in C^\infty({\Omega^d})$ gives rise to the equilibrium density $\rho^\star\in C^\infty({\Omega^d})$. Since the goal is to determine the marginals of $\rho^\star\in C^\infty({\Omega^d})$, we further assume that for $i,i'\in[d]$, $\rho_i\in C^\infty(\Omega),\rho_{ii'}\in C^\infty(\Omega^2)$. To facilitate the discussion, we define $L_i:C^\infty(\Omega)\times C^\infty(\Omega^2)^{d-1}\rightarrow C^\infty(\Omega)$ as
\begin{equation}\label{eq:one row}
L_i(\rho_i,\{\rho_{ii'}\}_{i'\in[d]\setminus i}) := -\frac{1}{\beta}  \frac{\partial^2\rho_i(x_i)}{\partial x_{i}^2}- \frac{\partial}{\partial x_{i}} \left(\sum_{i'\in[d]\setminus i}\int  \rho_{i i'}(x_i,x_{i'})  \frac{\partial V_{ii'}}{\partial x_{i}}(x_{i},x_{i'})dx_{i'}\right),\ i\in[d],
\end{equation}
$T_i: C^\infty(\Omega^2)^{d-1}\rightarrow C^\infty(\Omega)$ as
\begin{equation}\label{eq:transport operator}
T_i(\{\rho_{ii'}\}_{i'\in[d]\setminus i}) := - \frac{\partial}{\partial x_{i}} \left(\sum_{i'\in[d]\setminus i}\int  \rho_{i i'}(x_i,x_{i'})  \frac{\partial V_{ii'}}{\partial x_{i}}(x_{i},x_{i'})dx_{i'}\right),\ i\in[d],
\end{equation}
and $F_i: C^\infty(\Omega^2)^{d-1}\rightarrow C^\infty(\Omega)$ as
\begin{equation}\label{eq:force operator}
F_i(\{\rho_{ii'}\}_{i'\in[d]\setminus i}) :=  -\left(\sum_{i'\in[d]\setminus i}\int  \rho_{ii'}(x_i,x_{i'})\frac{\partial V_{ii'}}{\partial x_{i}}(x_{i},x_{i'})dx_{i'}\right),\ i\in[d].
\end{equation}
One can think of $F_i$ as some ``effective force'' exerted on the $i$-th dimension.

We now make a few assumptions needed for the proof. For the Fokker-Planck operator $L$, it can be shown that for some function $s$,
\begin{equation}\label{eq:FK identity 1}
L(\rho^\star s) = \rho^\star L s
\end{equation}
and furthermore 
\begin{equation}\label{eq:FK identity 2}
\langle t, \rho^\star L s\rangle = D(t,s),\quad D(t,s) := \int \nabla t(x) \cdot  \nabla s(x) \rho^\star(x)dx.
\end{equation}
(See for example \cite{pavliotis2016stochastic}.) We now make assumptions on the bilinear form $D$ (also called the Dirichlet's form).
\begin{assumption}\label{assump:coercivity}
The bilinear form $D$ satisfies
\begin{equation}
\lambda_1 \|s\|^2_{L^2(\rho^\star)} \leq D(s,s)=:\|s\|_D^2
\end{equation}
for any $s$ that is not a constant $\rho^\star$-almost everywhere.
\end{assumption}
This assumption on the lower-bound of the bilinear form is a standard ``spectral gap'' assumption  for showing the convergence of a time-dependent Fokker-Planck equation \cite{markowich2000trend}. The second assumption concerns the existence of a meanfield approximation.
\begin{assumption}\label{assump:meanfield1}
We assume that the probability density $\rho^\star\propto \exp(-\beta V)$ can be approximated by a meanfield, i.e.
\begin{equation}
 \frac{\prod_{i=1}^d \rho^\star_i(x_i)}{\rho^\star(x)} =    1+\epsilon(x) ,\quad \|\epsilon\|_{H^1(\rho^\star)}\leq \eta_1.
\end{equation}
\end{assumption}
This assumption is rather natural, since we first need the equilibrium density to be well approximated by a meanfield, before trying to determine the components of the meanfield  computationally.

\subsection{Main theorem}
The goal is to show the solution of the following convex problem
\begin{eqnarray}\label{eq:problem for proof}
 &\ &L_i(\rho_i,\{\rho_{ii'}\}_{i'\in[d]\setminus i})  = 0,\quad i\in [d]\cr 
&\ &\int\rho_{ii'}(x_i,x_{i'})dx_{i'} = \rho_i(x_i),\ \rho_{ii'}\geq 0, \ i'\in[d]\setminus i,\quad \int \rho_i(x_i)dx_i=1,\ \forall i\in[d],\cr 
 &\ &\sum_{i=1}^d\int s_i(x_i)^2\rho_i(x_i)dx_i - \sum^d_{i\neq i'=1}\int \rho_{ii'}(x_i,x_{i'}) s_i(x_i) s_{i'}(x_{i'})dx_idx_{i'}\geq 0,\quad \forall s_1,\cdots,s_d\in\mathcal{S}\cr 
 &\ &\frac{\rho_i}{\rho_i^\star},\frac{\rho_{ii'}}{\rho_{i}^\star \rho_{i'}^\star}\in \mathcal{S}, \forall i, i'\in[d], i\neq i'
\end{eqnarray}
recovers the meanfields  $\rho^\star_i,i\in[d]$. Here, the set $\mathcal{S}$ is defined to be
\begin{equation}\label{eq:regularity}
\mathcal{S}:=\{s\ \vert \left\| s\right\|_{L^2(\rho^\star)}\leq  \infty,\  \left\| \nabla s\right\|_{L^2(\rho^\star)}\leq R \|s\|_{L^2(\rho^\star)}, R>0\}.
\end{equation}
By an abuse of notation, in this section we treat a single variable function $s_i$ acting on the $i$-th dimensional variable $x_i$ as a $C^\infty(\Omega)$ and $C^\infty(\Omega^d)$ function interchangeably. In the latter case, we think of $s_i(x_i),x_i\in \Omega$ as $s_i(x) = s_i(x_i),x\in \Omega^d$. Although it seems like the third set of constraints has infinite number of constraints, it is equivalent to the positive semidefinite constraints in \eqref{eq:PSD_constraint} (with a suitable discretization using basis in $\mathcal{S}$), and therefore can be enforced with semidefinite programming.

In order to prove our main theorem, we make another assumption, specifically pertaining the convex relaxation \eqref{eq:problem for proof}. This assumption is relatively unusual in the literature of Fokker-Planck equation.
\begin{assumption}\label{assump:meanfield2}
We assume that $F_i$ defined \eqref{eq:force operator} satisfies
\begin{equation}
\sum_{i=1}^d \left\langle \frac{\partial s_i}{\partial x_i}, F_i(\{\rho_{ii'}-\rho_{i}\rho_{i'}\}_{i'\in[d]\setminus i}) \right\rangle \leq  \eta_2 \sqrt{\sum^d_{i=1}\left\|\frac{\rho_i}{\rho_i^\star} \right\|^2_{L^2(\rho^\star)}}  \sqrt{\sum^d_{i=1}\left\|s_i \right\|^2_{D}} ,\quad i\in [d]
\end{equation}
for any $\{\rho_{ii'}\}_{i,i'\in[d],i\neq i'}$ satisfying the constraints in \eqref{eq:problem for proof}.
\end{assumption}
This assumption essentially says the ``effective force'' computed from any 2-marginals $\{\rho_{ii'}\}_{i'\in [d]\setminus i}$ is not that different from treating each 2-marginal as independent, i.e. $\rho_{ii
} = \rho_i\rho_{i'}$. This depends on the nature of $V$, and also on the constraints in \eqref{eq:problem for proof}. In some sense, one can think of $\eta_2$ as ``relaxation error''. In Section~\ref{section:example}, we give a non-trivial example with strong pairwise interactions to show how  this assumption can hold without separability. 

We are now ready to present the main theorem.
\begin{theorem}\label{theorem:main}
Under Assumption~\ref{assump:coercivity},\ref{assump:meanfield1},\ref{assump:meanfield2}, the solution to the marginal relaxation problem \eqref{eq:problem for proof}
satisfies
\begin{equation}
  \sqrt{\sum_{i=1}^d \left\| \frac{\rho_i}{\rho_i^\star}-1\right\|^2_{L^2(\rho_i^\star)}}\leq \frac{ \sqrt{2}\eta_1 +\eta_2  }{\lambda_1(1-\sqrt{2}\eta_1)}\sqrt{\sum^d_{i=1}\left\|\frac{\rho_i}{\rho_i^\star} \right\|^2_{L^2(\rho_i^\star)}},\quad i\in[d].
\end{equation}
\end{theorem}
\begin{proof}
For $\rho_i,\{\rho_{ii'}\}_{i'\in[d]\setminus i}$ satisfying the constraints in \eqref{eq:problem for proof}, we have
\begin{eqnarray}
0&= &\sum_{i=1}^d \langle s_i, L_i(\rho_i,\{\rho_{i} \rho_{i'}\}_{i'})\rangle  + \langle s_i,  T_i(\{\rho_{ii'}-\rho_i\rho_{i'}\}_{i'\in[d]\setminus i})\rangle\cr
&=&\sum_{i=1}^d\langle s_i, L_i(\rho_i,\{\rho_{i} \rho_{i'}\}_{i'})\rangle  - \left\langle \frac{\partial s_i}{\partial x_i},   F_i(\{\rho_{i{i'}}-\rho_i\rho_{i'}\}_{{i'}\in[d]\setminus i})\right\rangle\cr 
&=& \sum_{i=1}^d\left \langle s_i, L\left(\prod^d_{{i'}=1}\rho_{i'}\right)\right \rangle - \left\langle \frac{\partial s_i}{\partial x_i},  F_i(\{\rho_{i{i'}}-\rho_i\rho_{i'}\}_{{i'}\in[d]\setminus i})\right\rangle\cr 
&=& \sum_{i=1}^d\left \langle s_i, \rho^\star L\left(\frac{\prod^d_{{i'}=1}\rho_{i'}}{\rho^\star}\right)\right \rangle  - \left\langle \frac{\partial s_i}{\partial x_i},   F_i(\{\rho_{i{i'}}-\rho_i\rho_{i'}\}_{{i'}\in[d]\setminus i})\right\rangle\cr 
&=& \sum_{i=1}^d\left \langle  \frac{\partial s_i}{\partial x_i} , \frac{\partial}{\partial x_i}\left( \frac{\prod^d_{{i'}=1}\rho_{i'}}{\rho^\star}\right) \rho^\star \right \rangle  - \left\langle \frac{\partial s_i}{\partial x_i},  F_i(\{\rho_{i{i'}}-\rho_i\rho_{i'}\}_{{i'}\in[d]\setminus i})\right\rangle\cr 
&=& \sum_{i=1}^d \left \langle  \frac{\partial s_i}{\partial x_i}, \frac{\partial}{\partial x_i}\left( \frac{\rho_i}{\rho^\star_i}\right) \rho_i^\star\right \rangle +  \left \langle \frac{\partial s_i}{\partial x_i}, \frac{\partial }{\partial x_i} \left( \frac{\prod_{i=1}^d \rho_i}{\prod_{i=1}^d \rho_i^\star}\epsilon\right)  \rho^\star \right \rangle - \left\langle \frac{\partial s_i}{\partial x_i},   F_i(\{\rho_{ij}-\rho_i\rho_{i'}\}_{{i'}\in[d]\setminus i})\right\rangle\cr 
&\geq &\sum_{i=1}^d D\left(  s_i,\frac{\rho_i}{\rho^\star_i}\right)- \sum_{i=1}^d\left\langle \frac{\partial s_i}{\partial x_i},    F_i(\{\rho_{i{i'}}-\rho_i\rho_{i'}\}_{{i'}\in[d]\setminus i})\right\rangle- \sqrt{\sum_{i=1}^d \|s_i\|^2_D} \left\|\frac{\prod_{i=1}^d \rho_i}{\prod_{i=1}^d \rho_i^\star}\epsilon\right\|_{D}\cr
&\geq &\sum_{i=1}^d D\left(  s_i,\frac{\rho_i}{\rho^\star_i}\right)-\sqrt{\sum_{i=1}^d \|s_i\|^2_D}  \left\|\frac{\prod_{i=1}^d \rho_i}{\prod_{i=1}^d \rho_i^\star}\epsilon\right\|_{D}-  \eta_2 \sqrt{\sum_{i=1}^d \|s_i\|^2_D}\sqrt{\sum^d_{i=1}\left\|\frac{\rho_i}{\rho_i^\star} \right\|^2_{L^2(\rho^\star)}} \label{eq:theorem part 1} 
\end{eqnarray}
The first equality follows from the first constraints of \eqref{eq:problem for proof}. The second equality follows from integration by parts and the definitions in \eqref{eq:transport operator} and \eqref{eq:force operator}. The fourth and fifth equality is due to the property of the Fokker-Planck operator \eqref{eq:FK identity 1} and \eqref{eq:FK identity 2}. The sixth equality is based on letting $\prod_{i=1}^d \rho^\star_i = \rho^\star(1 +\epsilon)$. The first  inequality is due to Cauchy-Schwarz, and the last inequality is due to Assumption~\ref{assump:meanfield2}.
Now, let 
\begin{equation}
s_i = \frac{\rho_i}{\rho^\star_i} -1
\end{equation}
from \eqref{eq:theorem part 1} we get 
\begin{eqnarray}
0&\geq & \sum_{i=1}^d \|s_i\|_D^2- \sqrt{2\eta_1^2}\sqrt{\sum_{i=1}^d \|s_i\|^2_D} \sqrt{\sum_{i=1}^d\|s_i+1\|^2_{L^2(\rho^\star)}+\|s_i\|^2_{D}}-\sqrt{\sum_{i=1}^d \|s_i\|^2_D} \sqrt{\sum^d_{i=1}\left\|\frac{\rho_i}{\rho_i^\star} \right\|^2_{L^2(\rho^\star)} }\eta_2\cr
&\geq& (1-\sqrt{2}\eta_1) \sum_{i=1}^d \|s_i\|_D^2 -\sqrt{2}\eta_1\sqrt{\sum_{i=1}^d \|s_i\|^2_D} \sqrt{\sum_{i=1}^d \left\| s_i+1\right\|^2_{L^2(\rho^\star)}}-\sqrt{\sum_{i=1}^d \|s_i\|^2_D}\sqrt{\sum^d_{i=1}\left\|\frac{\rho_i}{\rho_i^\star} \right\|^2_{L^2(\rho^\star)}} \eta_2\cr
\end{eqnarray}
where we apply Assumption~\ref{assump:meanfield1} in the first inequality. Combining this with Assumption~\ref{assump:coercivity} we get the conclusion
\begin{equation}
\lambda_1(1-\sqrt{2}\eta_1) \sqrt{\sum_{i=1}^d \left\| \frac{\rho_i}{\rho_i^\star}-1\right\|^2_{L^2(\rho_i^\star)}}\leq \sqrt{2}\eta_1\sqrt{\sum_{i=1}^d \left\| \frac{\rho_i}{\rho_i^\star}\right\|^2_{L^2(\rho_i^\star)}}+\eta_2\sqrt{\sum^d_{i=1}\left\|\frac{\rho_i}{\rho_i^\star} \right\|^2_{L^2(\rho_i^\star)}}.
\end{equation}
\end{proof}

\subsection{Example}\label{section:example}

At the first sight, it seems like the only pairwise potential that satisfies Assumption~\ref{assump:meanfield2} with a small enough $\eta_2$ is that each $V_{ii'}$ has the form $V_{ii'}(x_i,x_{i'}) = E_i(x_i)+E_{i'}(x_{i'})$, i.e. a separable potential. However, in this section we argue that it is possible to have non-separable $V_{ii'}(x_i,x_i') = E_i(x_i)+E_{i'}(x_i') + \delta_{ii'}(x_i,x_i')$ where $\delta_{ii'}$ is much larger compare to $E_i, E_{i'}$. Let $E_i=E$ for all $i\in [d]$, and further
\begin{equation} \frac{\partial}{\partial x_i}\delta_{ii'}(x_i,x_j) 
:=\sum_{k_i, k_{i'}\in [n]}a_{ii'}(k_i,k_{i'})\phi_{k_i}(x_i)\phi_{k_{i'}}(x_{i'})
\end{equation}
where each $a_{ii'}\in \mathbb{R}^{n\times n}$, and $\{\phi_j\}_{j=1}^n\subset L^2(\rho^\star)$ is a set of orthonormal single variable basis. We let $a_{ii'}(k_i,k_{i'})\sim \mathcal{N}(0,\sigma^2)$. In what follows, we show that one can have $\sigma$ as large $o(\sqrt{d})$, and still $\eta_2/\lambda_1\sim o(1)$, giving rise to a meaningful upper-bound in Theorem~\ref{theorem:main}. In this case, the term $\delta_{ii'}$ completely dominates $E_i,E_{i'}$ when $d$ is large.

To calculate $\eta_2$ in Assumption~\ref{assump:meanfield2} for this setting, we first notice that
\begin{eqnarray}\label{eq:bounding assumption3}
\sum_{i'\in [d]\setminus i} \left\langle t_i\frac{\partial V_{ii'}}{\partial x_i}, \rho_{ii'}-\rho_i\rho_j\right\rangle  &=& \sum_{i'\in [d]\setminus i} \left\langle t_i\frac{\partial \delta_{ii'}}{\partial x_i}, \rho_{ii'}-\rho_i\rho_{i'}\right\rangle\cr
&=&   \sum_{i'\in [d]\setminus i}  \langle\tilde t_i a_{ii'},  G_{ii'}\rangle - \sum_{i'\in [d]\setminus i}  \langle\tilde t_i a_{ii'},  \rho_{i}\rho_{i'}\rangle
\end{eqnarray}
where $ G_{ii'}\in \mathbb{R}^{n\times n},  \tilde t_i\in \mathbb{R}^{n\times n}$ are formed by discretizing with $\{\phi_i\}_{i=1}^n$. In particular
\begin{equation}
G_{ii'} = \left[\int \phi_j(x_i) \rho_{ii'}(x_i,x_i') \phi_{j'}(x_i')dx_idx_i'\right]_{j,j'\in [n]},\quad G_{ii'} = \text{diag}\left(\left[\int\phi_j(x_i) \phi_j'(x_i)\rho_{i}(x_i) dx_i\right]_{j,j'\in[n]}\right),
\end{equation}
$i,i'\in[d]$, and
\begin{equation}
\tilde t_i = \left[\int t_i(x_i) \phi_j(x_i)  \phi_{j'}(x_i)\rho_i^\star(x_i)dx_i\right]_{j,j'\in [n]}.
\end{equation}
One can think of $G=[G_{ii'}]_{i,i'\in [d]}$ as a size $n d \times nd$ matrix. It has several properties, in particular, $G\succeq 0$ due to the constraint in \eqref{eq:problem for proof}. Further $\tr(G_{ii'})\leq n$. We can factorize $G = Y Y^T$ where $Y =  [Y_1^T,\cdots, Y_d^T]^T$ and each $Y_i\in \mathbb{R}^{n\times n}$, and further $\|Y_i\|_F^2 = \tr(G_{ii}) \leq n$. Now we have
\begin{eqnarray}\label{eq:bounding spectral norm}
\sum_{i'\in [d]\setminus i}  \langle\tilde t_i a_{ii'},  G_{ii'}\rangle &=&\langle [a_{ii'}]_{i,i'\in[d]} Y, \text{diag}([\tilde t_i]_{i\in[d]})Y\rangle \cr 
&\leq& \| [a_{ii'}]_{i,i'\in[d]} \|_2 \|Y\|_F \sqrt{\sum_{i\in [d]} \|\tilde t_iY_i\|^2_F}\cr 
&\leq & O(\sigma\sqrt{nd} )\|Y\|_F \sqrt{\sum_{i\in [d]} \|\tilde t_i\|^2_F}\cr
&\leq & O(\sigma nd )\sqrt{\sum_{i\in [d]} \|t_i\|^2_{L^2(\rho^\star)}}\cr 
&\leq &O(\sigma n\sqrt{d})\sqrt{\sum_{i\in [d]} \|t_i\|^2_{L^2(\rho^\star)}}\sqrt{\sum_{i\in[d]} \left\| \frac{\rho_i}{\rho_i^\star}\right\|^2_{L^2(\rho^\star)}}\cr
\end{eqnarray}
for $\left\|\frac{\rho_i}{\rho^\star}\right\|_{L^2(\rho^\star)}\leq \text{constant}$ (Eq.~\eqref{eq:regularity}). The second inequality follows from standard spectral norm bounds on the $nd\times nd$ random Gaussian matrix $[a_{ii'}]_{i,i'\in[d]}$ \cite{vershynin2018high}, and the third inequality follows from  $\|Y_i\|_F\leq \sqrt{n}$.  In the following we treat $n$ as a constant. Similarly, by replacing $Y$ in the above derivation by 
\begin{equation}
Y\leftarrow [ [\langle \rho_i, \phi_j \rangle]_{j\in [n]}  ]_{i\in [d]} \in \mathbb{R}^{nd},
\end{equation}
we can also derive
\begin{equation}
\sum_{i'\in [d]\setminus i}  \langle\tilde t_i a_{ii'},  \rho_i\rho_{i'}\rangle \leq O(\sigma \sqrt{d})\sqrt{\sum_{i\in [d]} \|t_i\|^2_{L^2(\rho^\star)}}\sqrt{\sum_{i\in[d]} \left\| \frac{\rho_i}{\rho_i^\star}\right\|^2_{L^2(\rho^\star)}}.
\end{equation}
This gives an upperbound for \eqref{eq:bounding assumption3}, and by letting $t_i = \frac{\partial s_i}{\partial x_i}$, we obtain $\eta_2 = O(\sigma \sqrt{d})$ in Assumption~\ref{assump:meanfield2}.

We now compare $\eta_2$ to the eigenvalue $\lambda_1$ in Assumption~\ref{assump:coercivity}. Let $\lambda$ and $\psi$ and $\nu$ be the second lowest eigenvalue and its corresponding right and left eigenfunctions for the following equation
\begin{equation}
(L^\text{MF}_i\psi)(x_i) := -\frac{\partial^2 \psi(x_i)}{\partial x_i^2} - \frac{\partial}{\partial x_i}\left(\psi(x_i)\frac{\partial E(x_i)}{\partial x_i} \right) = \lambda \psi(x_i),\quad L_0^*\nu = \lambda \nu,\quad \nu = \psi/\rho^\text{MF}
\end{equation}
where $\rho^\text{MF} \propto \exp(-E)$. Letting $\beta = (d-1)^{-1}$, the operator $L$ becomes
\begin{equation}
L \rho =  (d-1)\sum_{i\in[d]} L^\text{MF}_i \rho+ \sum_{i\in[d]}\frac{\partial}{\partial x_i}\left(\sum_{i'\in[d]\setminus i} \frac{\partial \delta_{ii'}}{\partial x_{i'}} \rho\right).
\end{equation}
We now want to show this operator $L$ is not that different from $  (d-1)\sum_{i\in[d]} L^\text{MF}_i$. The second smallest eigenvalue of the first term $(d-1)\sum_{i\in[d]} L^\text{MF}_i$ is simply $(d-1)\lambda$, stemming from the fact that
\begin{equation}
 (d-1)\sum_{i\in[d]} L^\text{MF}_i \psi_1 = (d-1)\lambda  \psi_1.
\end{equation}
by choosing $\psi_1(x) = \psi(x_1)$. We now want to study the effect of perturbation on the eigenvalue caused by  $\delta_{ii'}$'s. Let $\nu_1(x) =  \nu(x_1)$. First order perturbation theory \cite{griffiths2018introduction} says the eigenvalue of $L$
\begin{eqnarray}
\lambda_1 \approx  (d-1)\lambda  + \sum_{i\in[d]}\left\langle \nu_1, \sum_{i\in [d]}\frac{\partial}{\partial x_i} \sum_{i'\in [d]\setminus i}\frac{\partial \delta_{ii'}}{\partial x_i}\psi_1\right\rangle
\end{eqnarray}
and further
\begin{eqnarray}
\sum_{i\in[d]}\left\langle \nu_1, \sum_{i\in [d]}\frac{\partial}{\partial x_i} \sum_{i'\in [d]\setminus i}\frac{\partial \delta_{ii'}}{\partial x_i}\psi_1\right\rangle&=& \left\langle \frac{\partial\nu}{\partial x_1},  \sum_{i'\in [d]\setminus 1}\frac{\partial \delta_{1i'}}{\partial x_1}\psi_1\right\rangle\cr
&=& \sum_{i'\in [d]\setminus 1}\left\langle \frac{\partial\nu}{\partial x_1}\frac{\partial \delta_{1i'}}{\partial x_1}, \psi \right\rangle\cr 
&\leq& O(\sigma \sqrt{d}) \left\| \frac{\partial\nu}{\partial x_1}\right\|_{L^2(\rho^\star)} \left\| \frac{\psi}{\rho^\star}\right\|_{L^2(\rho^\star)}\cr 
&\leq& O(\sigma \sqrt{d})\left\| \nu_1\right\|_{L^2(\rho^\star)}\left\|\psi_1\right\|_{L^2(\rho^\star)}\cr 
\end{eqnarray}
where the first inequality follows from $\text{Var}\left(\left\langle \frac{\partial\nu}{\partial x_1}\frac{\partial \delta_{1i'}}{\partial x_1}, \psi \right\rangle\right) \leq \left\| \frac{\partial\nu}{\partial x_1}\right\|^2_{L^2(\rho^\star)} \left\| \frac{\psi}{\rho^\star}\right\|^2_{L^2(\rho^\star)}$, due to the fact that $\left\langle \frac{\partial\nu}{\partial x_1}\frac{\partial \delta_{1i'}}{\partial x_1}, \psi \right\rangle$ is a sub-gaussian random variable \cite{vershynin2018high}. The second inquality follows from \eqref{eq:regularity} and $\nu_1=\nu$. By first order perturbation theory of the eigenvalue, we get $\lambda_1 \approx (d-1)\lambda +  O(\sigma \sqrt{d} )$. In this case
\begin{equation}
\eta_2/\lambda_1 = o(1)
\end{equation}
for $\sigma = o(\sqrt{d})$, leading to a meaningful upper-bound for Thm.~\ref{theorem:main}.

A similar perturbation analysis based on first order perturbation theory gives eigenfunction perturbation $\eta_1$ in Assumption~\ref{assump:meanfield1} where $\eta_1/\lambda_1 = o(1)$ if $\sigma = o(\sqrt{d})$. 

\section{Numerical Experiments}
\label{sec:numerical}

\subsection{Stationary solution} 
In this subsection, we demonstrate the numerical performance of solving the stationary FPE via cluster moment relaxation. We consider two systems in this subsection, one with the separable double-well potential, which is intrinsically an 1D potential so we can easily visualize, and the other one with the Ginzburg-Landau potential, for which we can only try to compare the moments or marginals since the true density is exponentially sized.

\subsubsection{Double-well Potential}

We consider the following double-well potential function
\begin{align}
    V(x) = (x_1^2 - 1)^2 + 12\sum_{j=2}^d x_j^2, 
    \label{eq:DW-potential}
\end{align}
which essentially only contains one-body terms and
can be easily separable. The equilibrium distribution for this is simply
\begin{align}
    \rho^\star(x) = \frac{1}{Z_{\beta}} \exp \left(-\beta (x_1^2 - 1)^2 \right) \prod_{j=2}^d \exp \left(-12\beta x_j^2 \right),
\end{align}
where $Z_{\beta}$ is the partition function. The separable structure 
gives us a convenient way to visualize the distribution via the 1-marginals. In this example, we take $\beta=2.5$ and a $d=10$ dimensional system. 

We use monomials as test functions for this example. More specifically, we let the single variable basis set to be monomials, i.e. $\{\phi_j(\cdot)=(\cdot)^{j-1}\}_{j=1}^{9}$ and we take the test functional space to be the induced $1$-body functional space $\mathcal{T} = \mathcal{F}_1$. The potential function $V$ can be exactly expanded by monomials up to order $4$. For convenience, we denote the moment matrix $M$ in \eqref{eq:moment_constraint} as second-order moments with notation $M_2$ and introduce the first-order moment vector $M_1=\text{vec}(M_{\I \J, \I'=\emptyset \J'=0})$ to be a subset of $M_2$. We solve the feasibility semidefinite programming problem \eqref{eq:moment_constraint} with CVX \cites{grant2008graph,grant2014cvx}.

We compare the estimated moments $\hat{M}_1$ and $\hat{M}_2$ from the proposed semidefinite relaxation with the ground truth moments. We measure the relative errors in first and second order moments, defined by
\begin{align}
    E_1 = \frac{\|M_1 - \hat{M}_1\|_2}{\|M_1\|_2}, E_2 = \frac{\|M_2 - \hat{M}_2\|_{\text{F}}}{\|M_2\|_{\text{F}}},
\end{align}
respectively. In our example, we achieve $E_1=5.30\times 10^{-3}$ and $E_2=7.84\times 10^{-3}$ error for first and second order moments.

Next, we can estimate the 1-marginals from the estimated moments by solving the following maximum entropy problem
\begin{align}
    \ &\max_{\rho_i}\ \int \rho_i(x_i)\ln(\rho_i(x_i))dx_i\nonumber\\
    \text{s.t.} \ & \int \rho_i(x_i) \phi_j(x_i) dx_i = M_{\I=\{i\} \J=\{j\}, \I'=\emptyset \J'=0},
    \label{eq:maxium_entropy}
\end{align}
where $\rho_i$ is the $i$-th marginal, which is a standard method for density estimation from moments~\cite{dudik2007maximum}.  We show the estimated $1$-marginal by solving a discretized program of \eqref{eq:maxium_entropy} and the true $1$-marginal in \figref{fig:DW_d10}.

\begin{figure}[htb]
    \centering
    \begin{subfigure}{0.50\textwidth}
        \centering
        \includegraphics[width=\textwidth]{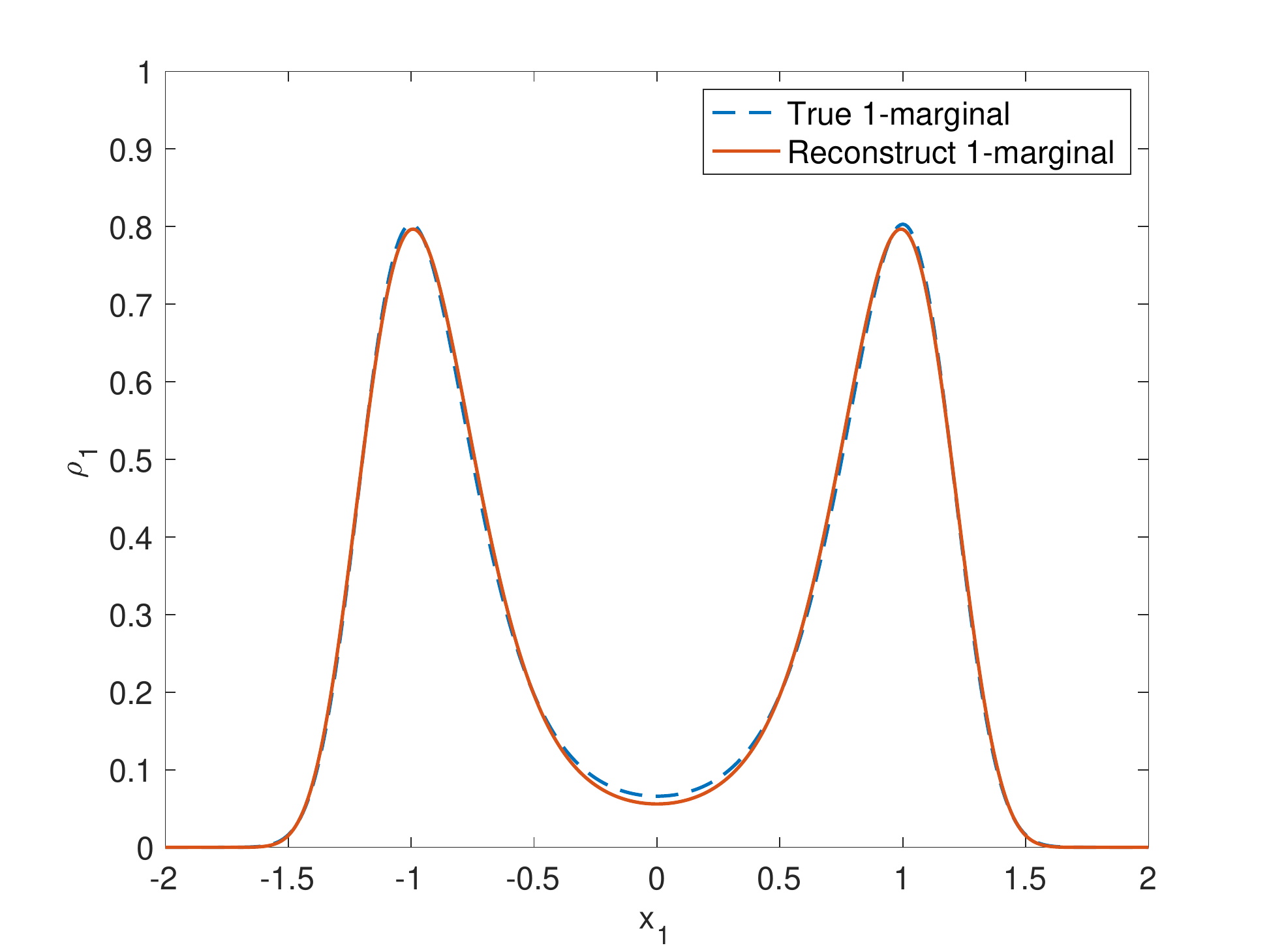}
    \end{subfigure}
    \caption{Visualization of ground truth $1$-marginals and reconstructed $1$-marginals for double-well potential system.}
    \label{fig:DW_d10}
\end{figure}

\subsubsection{Ginzburg-Landau Potential}

The Ginzburg-Landau theory was developed to provide a phenomenological description of many-body systems \cite{GL-model}. In this numerical example, we consider a periodic discretized Ginzburg-Landau model, 
\begin{align}
    V(U) := \sum_{i=1}^{d+1} \frac{\lambda}{2} \left(\frac{U_i - U_{i-1}}{h}\right)^2 + \frac{1}{4\lambda} (1 - U_i^2)^2,
    \label{eq:discrete-GL}
\end{align}
where $h = 1/(d+1)$ and $U_{d+1}=U_1$, $U_0=U_d$ based on periodic condition. We fix $d=10$, $\lambda=0.03$ and the temperature $\beta=1/16$. We let the computational domain to be $U\in [-1.6,1.6]^d$.

In this example, we demonstrate the performance of the proposed relaxation with complex Fourier basis. We note that the performance using monomial basis is similar. More specifically, we take $\{\phi_j(\cdot)=\exp(\iu \frac{\pi}{3.2}\cdot)^j\}_{j=-10}^{10}$ as univariate basis functions for each dimension which induce the test function space $\mathcal{T} = \mathcal{F}_1$. The potential function $V$ can be expanded with the same set of $21$ complex Fourier basis functions in each dimension with relative error $5.31\times 10^{-3}$.

Using the test function space discussed above, we achieve $E_1=1.55\times 10^{-2}$ and $E_2=5.11\times 10^{-1}$ relative error in first and second order moments. The proposed method is able to identify the first order moment accurately but the estimated second order moment has a relatively large error. We may need to use  to third order moments if we want to identify the second order moments to high accuracy, due to the deviation of Ginzburg-Landau model from a mean-field model.

Similarly we can solve the maximum entropy problem \eqref{eq:maxium_entropy} to reconstruct the estimated marginals from the moments for any dimension. Here we compare the reconstructed marginal (red solid line) and ground truth marginal (blue dashed line) in \figref{fig:GL_d10}.

\begin{figure}[htb]
    \centering
    \begin{subfigure}{0.50\textwidth}
        \centering
        \includegraphics[width=\textwidth]{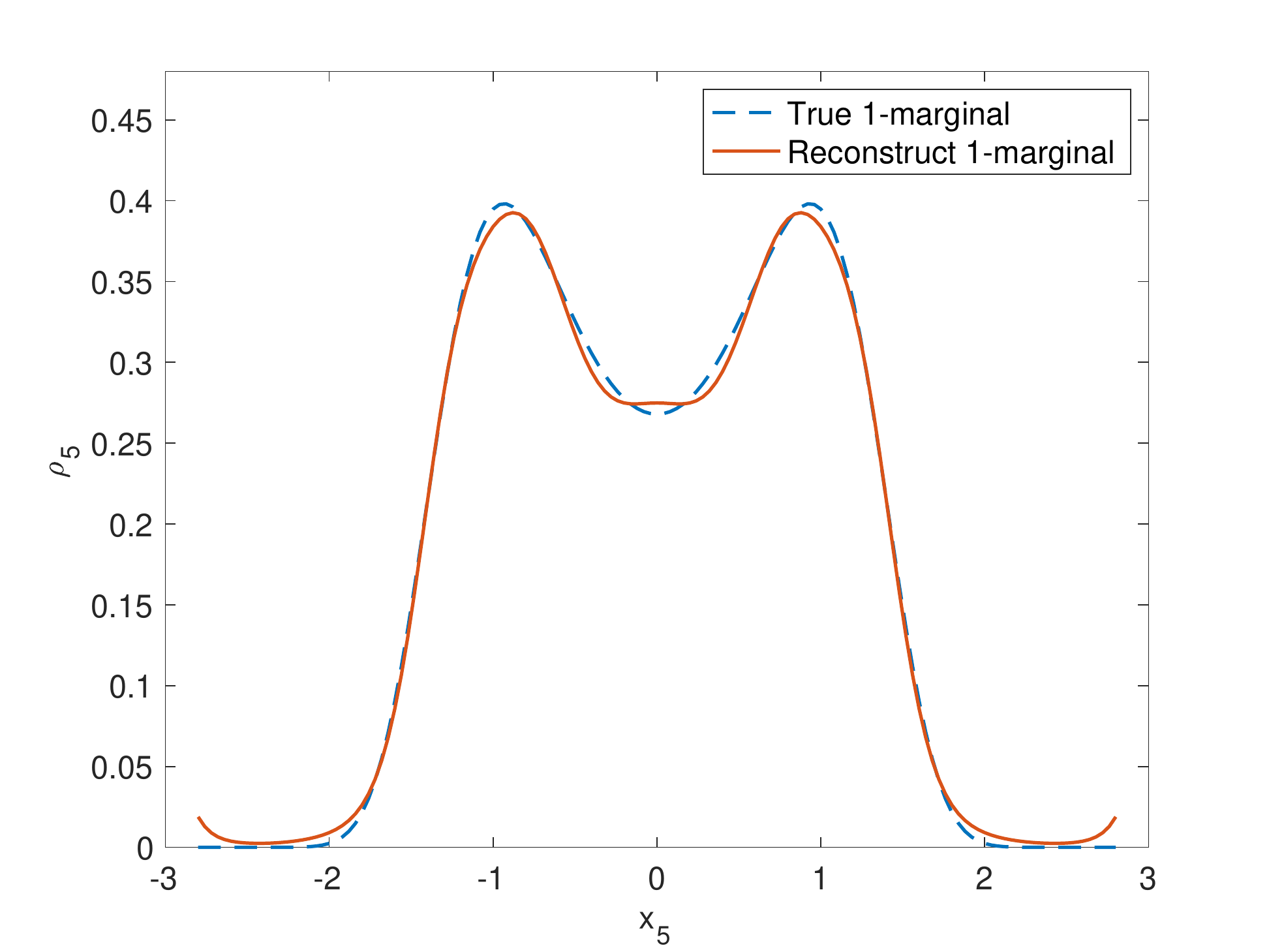}
    \end{subfigure}
    \caption{Visualization of ground truth $1$-marginals and reconstructed $1$-marginals for periodic Ginzburg-Landau potential system.}
    \label{fig:GL_d10}
\end{figure}

\subsection{Time-dependent problem} 

In this subsection, we demonstrate the numerical performance of solving the time-dependent FPE via marginal relaxation. More specifically, we consider the extension of our framework to time dependent Fokker-Planck equation as discussed in \secref{sec:time_dependent}.

We consider a double-well system \eqref{eq:DW-potential} with $d=2$ for better visualization purpose. For all spatial dimensions, we restrict the support of our marginals in hypercube $[-2,2]^2$ and use $100$ evenly distribution grid points to discretize each dimension. This means we consider the density $\rho$ to be negligible outside the hypercube $[-2,2]^2$, which is a property of systems with confining potential. The time dimension is discretized by $100$ grid points as well with a gap $\delta t=2\times 10^{-4}$. We solve the time dependent optimization problem \eqref{eq:key_evolving} with marginals of order at most $2$. 

We use the initial and boundary conditions in \eqref{eq:time dependent bc}. We assume the particles are concentrated around $(-1,0,0,0,0)$, one of the local minima of the double-well potential, at the beginning of evolution. More specifically, we let $\rho(x,0) = \rho_1(x_1,0)\prod^d_{i=2}\rho_i(x_i,0)$ where $\rho_1(x_1,0) = \text{Uniform}[-1,2,0.8]$, $\rho_i(x_i,0) \propto \exp(-0.3 \vert x_i \vert^2), i \neq 1$. We test two scenarios for absorbing states: (1) there is an absorbing state at $0.2$ and we require $\rho(0.2,\tau)=0$ for all $\tau\in [0, 0.02]$, and (2) there is an absorbing state at $1.0$ and we require $\rho(1.0,\tau)=0$ for all $\tau\in [0, 0.02]$. 

Since the double-well potential is separable, the ground truth $1$-marginals can be obtained by solving the time-dependent Fokker-Planck equation for the first dimension only,  
\begin{align}
   &\frac{\partial \rho_1(x_1,\tau)}{\partial \tau} = \frac{1}{\beta} \frac{\partial^2 \rho_1(x_1, \tau)}{\partial x_1^2} + 4x_1(x_1^2 - 1)\frac{\partial \rho_1(x_1,\tau)}{\partial x_1}, \notag\\ &\rho_1(x_1,0)=\text{Uniform}(-1.2,-0.8),\quad
   \rho_1(x_a,\tau) = 0,\ \forall \tau\in [0,0.02],
   \label{eq:truth_DW}
\end{align}
where $x_a=0.2,1.0$ corresponds to scenarios (1) and (2), respectively. 

We compare the estimated $1$-marginals with the ground truth marginals by solving \eqref{eq:truth_DW}. The average elementwise relative error for all $1$-marginals over time is (1) $1.02\times 10^{-1}$ and (2) $1.39\times 10^{-1}$ for the two scenarios mentioned above. We further visualize the evolution of $1$-marginals in \figref{fig:evolving1} and \figref{fig:evolving2}.

\begin{figure}[htb]
   \centering
   \begin{subfigure}{0.49\textwidth}
   \includegraphics[width=1.0\textwidth]{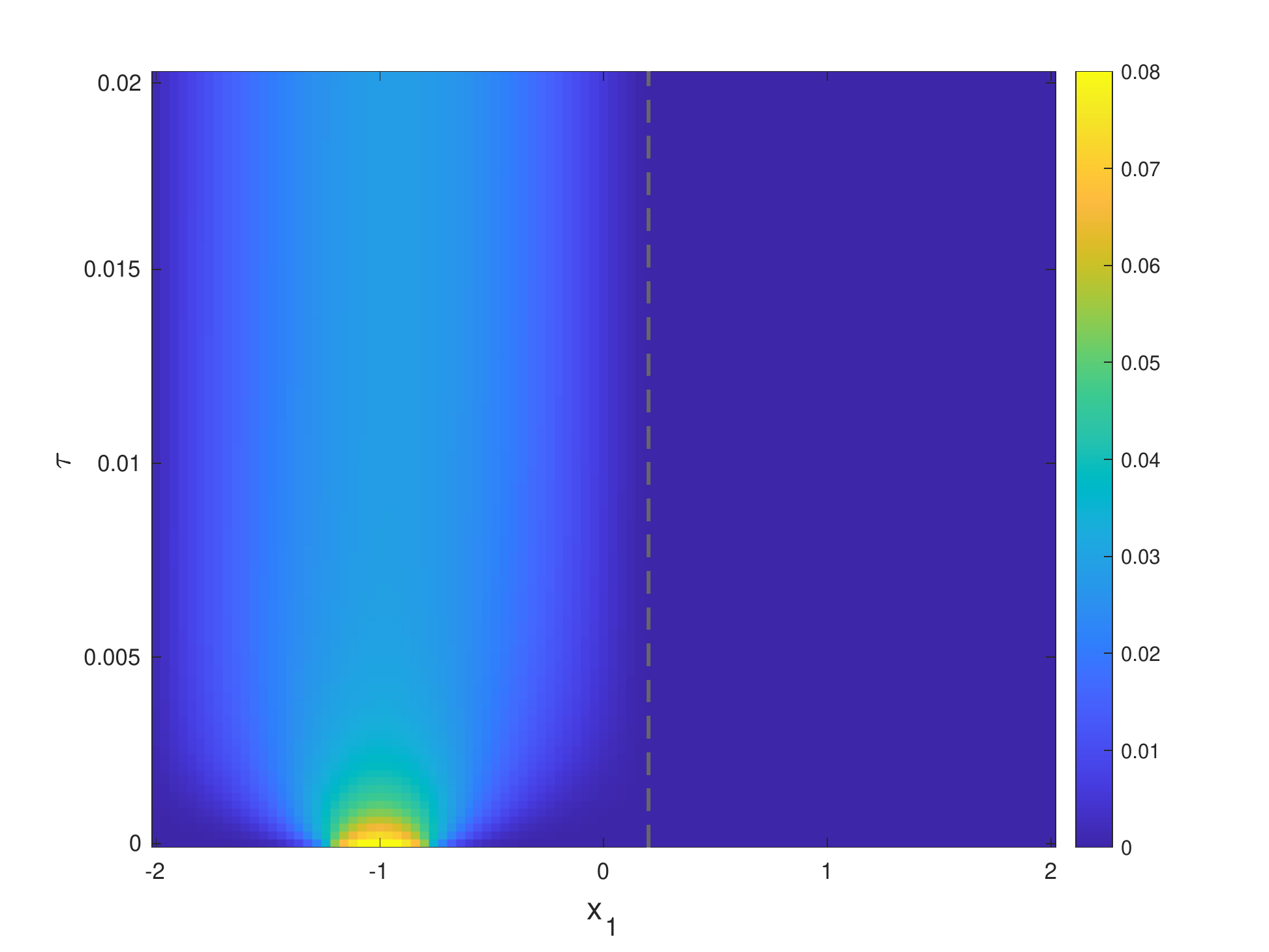}
   \caption{Ground truth $1$-marginals over time}
   \end{subfigure}
   \begin{subfigure}{0.49\textwidth}
   \includegraphics[width=1.0\textwidth]{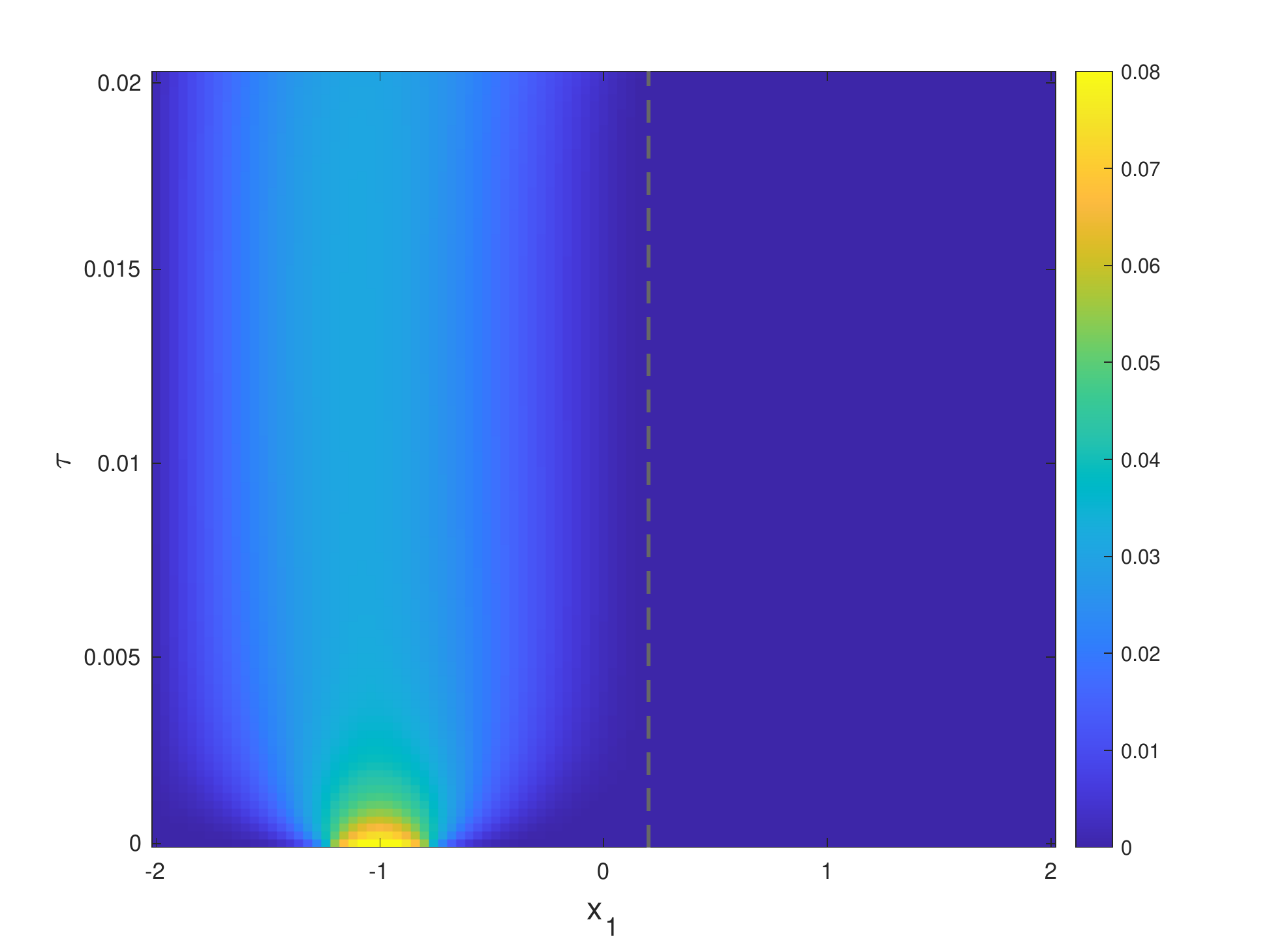}
   \caption{Estimated $1$-marginals over time}
   \end{subfigure}
   \caption{Comparison of the ground truth $1$-marginals and the estimated marginals over time. Here we assume there exists an absorbing state at $x_1=0.2$. We use dashed vertical line to indicate the absorbing state.}
   \label{fig:evolving1}
\end{figure}

\begin{figure}[htb]
   \centering
   \begin{subfigure}{0.49\textwidth}
   \includegraphics[width=1.0\textwidth]{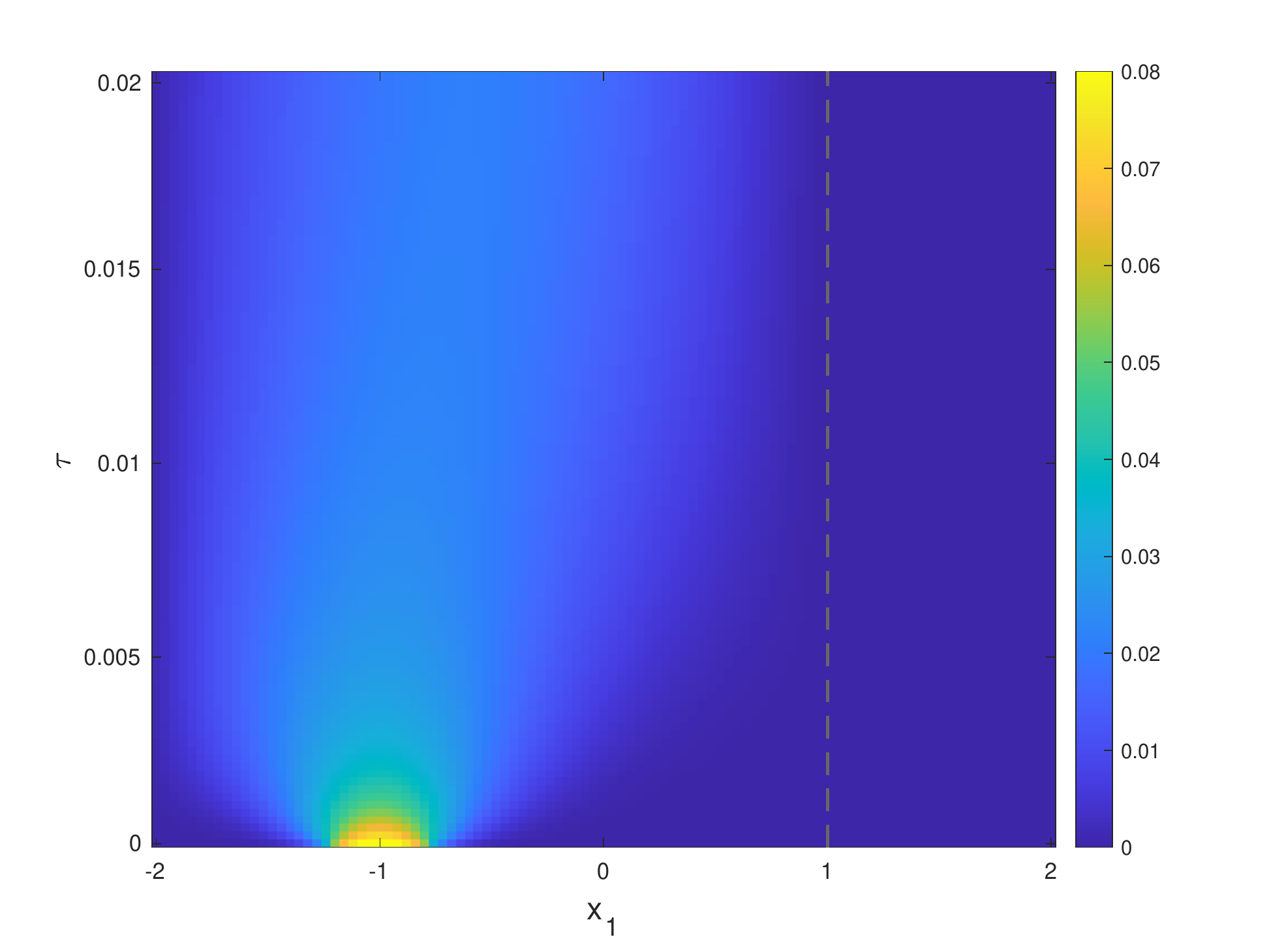}
   \caption{Ground truth $1$-marginals over time}
   \end{subfigure}
   \begin{subfigure}{0.49\textwidth}
   \includegraphics[width=1.0\textwidth]{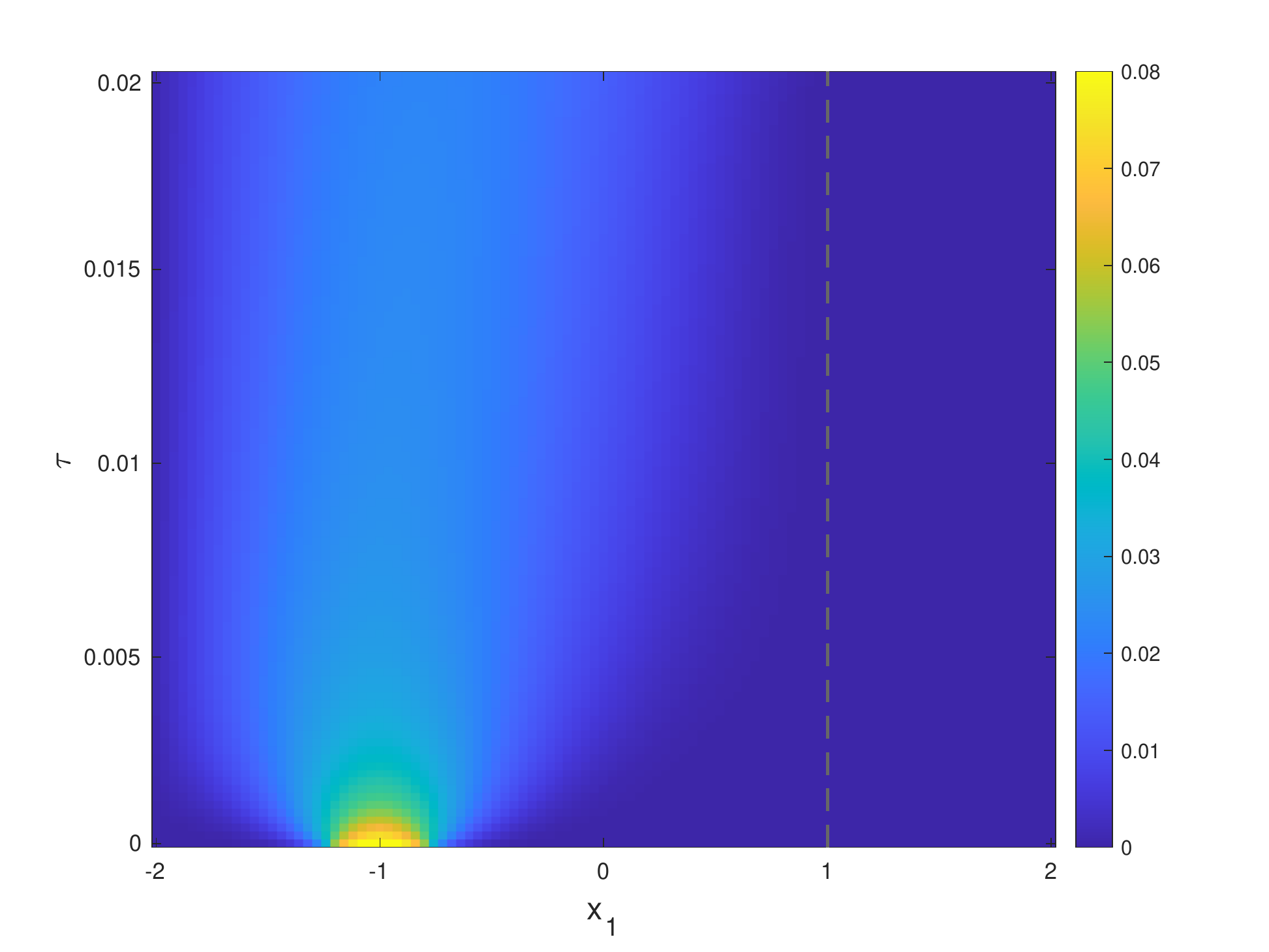}
   \caption{Estimated $1$-marginals over time}
   \end{subfigure}
   \caption{Comparison of the ground truth $1$-marginals and the estimated marginals over time. Here we assume there exists an absorbing state at $x_1=1.0$. We use dashed vertical line to indicate the absorbing state.}
   \label{fig:evolving2}
\end{figure}

\section{Conclusion} \label{sec:conclusion}

In this paper, we propose a novel algorithm to solve the low-order marginals/moments for high-dimensional equilibrium distributions in statistical mechanics via solving the Fokker-Planck PDE. We further reformulate the PDEs as constraints in our convex programs via cluster moments or marginal relaxations. The resulting optimization problem is fully deterministic and convex, which can be solved to desired accuracy with existing convex program solvers. 

By solving the low-order marginals/moments instead, we reduce the total number unknown parameters to estimate and as a result, we break the curse of dimensionality. The complexity of the algorithm depends on the number of dimensions and the number of marginals/moments to solve for.  Moreover our algorithm can be used potentially to provide a coarse description of the equilibrium distribution by providing a meanfield solution, and can be used to initialize other local optimization algorithms (e.g. MCMC or tensor-networks) for further refinement. Lastly, we also discuss generalization to time dependent Fokker-Planck equation which could be used to study non-equilibrium physics.

\section{Acknowlegement}
Y.C. and Y.K. acknowledge partial supports from NSF Award No.\ DMS-211563 and DOE Award No.\ DE-SC0022232.

\bibliographystyle{plain}
\bibliography{ref}

\end{document}